\documentclass[10pt,a4paper]{amsart}

\usepackage[all]{xy}
\usepackage[normalem]{ulem}
\usepackage{amsmath}
\usepackage{amsfonts}
\usepackage{amssymb}
\usepackage{amsthm}
\usepackage{graphicx}
\usepackage{microtype}
\usepackage[usenames,dvipsnames]{xcolor}
\usepackage{pgf,tikz}
\usepackage{mathrsfs}
\usetikzlibrary{arrows}
\usepackage{caption}
\usepackage{soul}
\usepackage[pdftex,plainpages=false,bookmarks=true,breaklinks=true,linkcolor=brown,citecolor=brown,colorlinks]{hyperref}
\usepackage{bm}
\usepackage{cancel} %to strikeout text
\usepackage{combelow}%for special character in romanian alphabet
\usepackage{here}

\newcommand{\prarrow}[2]{\ar@<0.5ex>[r]^-{#1} \ar@<-0.5ex>[r]_-{#2}}
\newcommand{\plarrow}[2]{\ar@<0.5ex>[l]^-{#1} \ar@<-0.5ex>[l]_-{#2}}
\newcommand{\pdarrow}[2]{\ar@<0.5ex>[d]^-{#1} \ar@<-0.5ex>[d]_-{#2}}
\newcommand{\puarrow}[2]{\ar@<0.5ex>[u]^-{#1} \ar@<-0.5ex>[u]_-{#2}}

\newcommand{\Rno}[1]{\textcolor{red}{no (#1)}}

\newtheorem{theorem}{Theorem}[section]
\newtheorem{corollary}[theorem]{Corollary}    
\newtheorem{lemma}[theorem]{Lemma}

\theoremstyle{definition}
\newtheorem{definition}[theorem]{Definition}   
    
\newtheorem{remark}[theorem]{Remark}
\newtheorem{problem}[theorem]{Problem}

\title{Whitney Properties and Whitney reversible properties of Cut and Non-Cut Points}

\author{Eiichi Matsuhashi}
\address{Department of Mathematics, Shimane University, Matsue, Shimane, 690-8504, Japan.}
\email{matsuhashi@riko.shimane-u.ac.jp}

\begin{document}

\begin{abstract}
In this paper, we study several properties concerning cut points and non-cut points. First, we show that the property of being a continuum consisting entirely of shore points is preserved under refinable maps. Next, we show that having a block point, having a non-shore point, and having a strong center point are not Whitney-reversible properties. Consequently, \cite[Question 2.9]{nonweak} has a negative answer. Furthermore, we show that none of the classes of continua witnessing the failure of the sequential strong Whitney-reversible property for these three properties admits a common model. Finally, in connection with the fact that colocal connectedness is not a Whitney-reversible property, we show that the class of nonaposyndetic continua having a cut point and such that every positive Whitney level is colocally connected has no common model.

\end{abstract}

\keywords{shore point, non-block point, non-strong center point,  colocally connected, Whitney property, Whitney-reversible property, common model, refinable map}
\subjclass[2020]{Primary  54F15 ; Secondary 54F16}      

\maketitle
\markboth{}{Eiichi Matsuhashi}

\section{Introduction}
 Throughout this paper, all spaces are assumed to be metrizable. A \emph{continuum} is a compact connected metric space. 

 %A topological property $P$ is called a \emph{Whitney property} if, whenever a continuum $X$ has property $P$, the level $\mu^{-1}(t)$ has property $P$ for every Whitney map $\mu$ for $C(X)$ and every $t\in[0,\mu(X))$. A topological property $P$ is called \emph{Whitney reversible} if, whenever $X$ is a continuum such that $\mu^{-1}(t)$ has property $P$ for every Whitney map $\mu$ for $C(X)$ and every $t\in(0,\mu(X))$, then $X$ itself has property $P$.

%A point $p$ of a continuum $X$ is called a \emph{cut point} of $X$ if $X\setminus \{p\}$ is disconnected, and a \emph{non-cut point} otherwise. A point $p$ of a continuum $X$ is called a \emph{shore point} of $X$ if, for every $\varepsilon>0$, there exists a subcontinuum $K$ of $X\setminus \{p\}$ such that $H_d(K,X)<\varepsilon$, where $H_d$ denotes the Hausdorff metric induced by a compatible metric $d$ on $X$. A point $p$ of a continuum $X$ is called a \emph{point of colocal connectedness} if, for every subcontinuum $K$ of $X$ containing $p$ in its interior, there exists a connected open subset $U$ of $X$ such that $p\in U\subset K$. A continuum all of whose points are points of colocal connectedness is called a \emph{colocally connected continuum}.

In continuum theory, examining how topological properties are preserved under mappings (such as refinable maps) and within hyperspaces (such as Whitney levels) provides two fundamental perspectives on structural stability. In this paper, we study the various types of non-cut points and cut points considered in \cite{noncut} and their behavior under refinable maps, and we also examine their connections with Whitney properties, Whitney-reversible properties, and the non-existence of common models.

The paper is organized as follows.

In Section 2, we present the terminology and definitions that will be used throughout the paper, together with some known results.

In Section 3, we consider topics concerning refinable maps. It was shown in \cite[Theorem 2.2]{refinablematsuhashi} that colocal connectedness is preserved under refinable maps. We establish the corresponding result for the property of being a continuum consisting entirely of shore points.

In Section 4, we consider topics concerning Whitney properties and Whitney-reversible properties. It was shown in \cite{refinablematsuhashi} that colocal connectedness is a Whitney property but is not a Whitney-reversible property. It was also shown in \cite[Corollary 2.4]{nonweak} that having a weak cut point is a sequential strong Whitney-reversible property, whereas \cite[Theorem 2.2]{matsuhashicut} showed that having a cut point is not a Whitney-reversible property. We demonstrate that having a block point, having a non-shore point, and having a strong center point are not Whitney-reversible properties. Consequently, \cite[Question 2.9]{nonweak} is answered in the negative. %The known results concerning these point properties and their relations to Whitney properties, Whitney-reversible properties, strong Whitney-reversible properties, and sequential strong Whitney-reversible properties are summarized in a table.

In Section 5, we study the non-existence of common models for classes of continua related to Whitney properties and Whitney-reversible properties. Related questions concerning common models have also been considered in \cite{none}.  We show that none of the classes of continua witnessing the failure of the sequential strong Whitney-reversible property for the properties of having  a block point, having a non-shore point, and having a strong center point admits a common model. Finally, motivated by the fact that colocal connectedness is not a Whitney-reversible property, we show that the class of nonaposyndetic continua having a cut point and such that every positive Whitney level is colocally connected has no common model.

\section{Preliminaries}

In this section, we introduce the terminology and definitions used throughout the paper and recall some known results.

An \emph{arc} is a space homeomorphic to the closed interval $[0,1]$. A continuum is called \emph{decomposable} if it is the union of two proper subcontinua, and \emph{indecomposable} otherwise. A continuum is called \emph{hereditarily indecomposable} if every subcontinuum of it is indecomposable. For a space $X$ and a subset $A\subset X$, we denote the interior of $A$ in $X$ by $\mathrm{Int}_X A$ and the closure of $A$ in $X$ by $\mathrm{Cl}_X A$.

If $(X,d)$ is a space and $A\subseteq X$, we write $\mathrm{diam}_d A$ for $\sup \{d(a,b)\mid a,b\in A\}$. Let $(X,d_X)$ and $(Y,d_Y)$ be continua and let $\varepsilon>0$. A surjective map $f:X\to Y$ is called an $\varepsilon$-\emph{map} if $\mathrm{diam}_{d_X}f^{-1}(y)<\varepsilon$ for every $y\in Y$. If $g,g_\varepsilon:X\to Y$ are surjective maps such that $g_\varepsilon$ is an $\varepsilon$-map and $d_Y(g(x),g_\varepsilon(x))<\varepsilon$ for every $x\in X$, then $g_\varepsilon$ is called an $\varepsilon$-\emph{refinement} of $g$. A surjective map $r:X\to Y$ is called a \emph{refinable map} if, for every $\varepsilon>0$, it admits an $\varepsilon$-refinement $r_\varepsilon:X\to Y$. The notion of a refinable map was introduced in \cite{refinable}. A typical example, as described in the introduction of \cite{refinable}, is the map from the $\sin\frac{1}{x}$-continuum onto an arc obtained by shrinking the limit bar to a point. In that paper, several theorems concerning topological properties preserved under refinable maps are proved.
%The same paper also shows that, if there is a refinable map $r$ from a continuum $X$ onto a continuum $Y$, then $X$ is decomposable if and only if $Y$ is decomposable (\cite[Corollary 1.3]{refinable}). Moreover, a map $r$ from an arc-like continuum onto an arc is refinable if and only if $r$ is monotone (\cite[Theorem 5]{refinable}).

\smallskip

Let $(X,d)$ be a continuum and let $x\in X$. The following six properties of $x$ have been well studied in \cite{noncut}.

\smallskip

$\mathbf{P_1}$: $x$ is a \emph{point of colocal connectedness} if for every neighborhood $U$ of $x$, there exists a neighborhood $V$ of $x$ with ${\rm Cl}_X V\subset U$ such that $X\setminus V$ is connected. In particular, if every point of $X$ has $\mathbf{P_1}$, then $X$ is called a \emph{colocally connected continuum}.

$\mathbf{P_2}$: $x$ is a \emph{non-weak cut point} if every pair of points in $X\setminus\{x\}$ is contained in a subcontinuum of $X\setminus\{x\}$.

$\mathbf{P_3}$: $x$ is a \emph{non-block point} if there exist subcontinua $A_1\subset A_2\subset\cdots\subset X$ such that $\bigcup_{n=1}^{\infty}A_n$ is dense in $X\setminus\{x\}$.

$\mathbf{P_4}$: $x$ is a \emph{shore point} if, for every $\varepsilon>0$, there exists $A\in C(X)$ such that $x\notin A$ and $H_d(A,X)<\varepsilon$, where $H_d$ denotes the Hausdorff metric on $C(X)$ induced by $d$.

$\mathbf{P_5}$: $x$ is a \emph{non-strong center point} if for every pair of nonempty open subsets $U,V$ of $X$, there exists a subcontinuum of $X\setminus\{x\}$ which intersects both $U$ and $V$.

$\mathbf{P_6}$: $x$ is a \emph{non-cut point} if $X\setminus\{x\}$ is connected.

\smallskip

These properties satisfy the following implications:
\[
\mathbf{P_1}\Longrightarrow\mathbf{P_2}\Longrightarrow\mathbf{P_3}\Longrightarrow\mathbf{P_4}\Longrightarrow\mathbf{P_5}\Longrightarrow\mathbf{P_6}.
\]
In general, none of the reverse implications is valid (see \cite{noncut}).

\smallskip

For each $i \in \{1, \ldots, 6\}$, we consider the following two properties of a continuum:
\begin{itemize}
\item $\mathcal{P}_i$: every point of the continuum has property $\mathbf{P}_i$, and
\item $\mathcal{Q}_i$: there exists a point of the continuum that does not have property $\mathbf{P}_i$.
\end{itemize}
Note that $\mathcal{Q}_i$ is precisely the negation of property $\mathcal{P}_i$. 

It follows directly from the definitions and the implications among $\mathbf{P}_i$ that these global properties satisfy the following relations:
\[
\mathcal{P}_1 \Longrightarrow \mathcal{P}_2 \Longrightarrow \mathcal{P}_3 \Longrightarrow \mathcal{P}_4 \Longrightarrow \mathcal{P}_5 \Longrightarrow \mathcal{P}_6
\]
and
\[
\mathcal{Q}_6 \Longrightarrow \mathcal{Q}_5 \Longrightarrow \mathcal{Q}_4 \Longrightarrow \mathcal{Q}_3 \Longrightarrow \mathcal{Q}_2 \Longrightarrow \mathcal{Q}_1.
\]

\medskip
For a continuum $X$, let $2^X$ denote the collection of all closed subsets of $X$, endowed with the topology induced by the Hausdorff metric. We use $C(X)$ for the subspace of $2^X$ consisting of all nonempty subcontinua of $X$. A \emph{Whitney map} is a continuous function $\mu:C(X)\to[0,\mu(X)]$ such that $\mu({x})=0$ for every $x\in X$ and $\mu(A)<\mu(B)$ whenever $A,B\in C(X)$ satisfy $A\subsetneq B$. It is known that, for every Whitney map $\mu:C(X)\to[0,\mu(X)]$ and every $t\in[0,\mu(X)]$, the level $\mu^{-1}(t)$ is a continuum (see \cite{illanes}).

\smallskip

Let $X$ be a continuum and let $\mu:C(X)\to[0,\mu(X)]$ be a Whitney map. Let $\mathcal{P}$ be a topological property of continua.

\begin{itemize}
\item $\mathcal{P}$ is called a \emph{Whitney property} (WP) if, whenever $X$ has $\mathcal{P}$, then $\mu^{-1}(t)$ has $\mathcal{P}$ for every Whitney map $\mu$ for $C(X)$ and every $t\in[0,\mu(X))$.

\item $\mathcal{P}$ is called a \emph{Whitney-reversible property} (WRP) if, whenever $\mu^{-1}(t)$ has $\mathcal{P}$ for every Whitney map $\mu$ for $C(X)$ and every $t\in(0,\mu(X))$, then $X$ has $\mathcal{P}$.

\item $\mathcal{P}$ is called a \emph{strong Whitney-reversible property} (SWRP) if, whenever there is a Whitney map $\mu$ for $C(X)$ such that $\mu^{-1}(t)$ has $\mathcal{P}$ for every $t\in(0,\mu(X))$, then $X$ has $\mathcal{P}$.

\item $\mathcal{P}$ is called a \emph{sequential strong Whitney-reversible property} (SSWRP) if, whenever there are a Whitney map $\mu$ for $C(X)$ and a sequence ${t_n}_{n=1}^{\infty}$ in $(0,\mu(X))$ with $t_n\to0$ such that $\mu^{-1}(t_n)$ has $\mathcal{P}$ for every $n$, then $X$ has $\mathcal{P}$.
\end{itemize}

For further information on these notions, see, for example, \cite{illanes}.

\smallskip

Note that if $\mathcal{P}_i$ is a Whitney property, then $\mathcal{Q}_i$ is a sequential strong Whitney-reversible property. Similarly, if $\mathcal{Q}_i$ is a Whitney property, then $\mathcal{P}_i$ is a sequential strong Whitney-reversible property.

\smallskip

A continuum $M$ is called a \emph{common model} for a class $\mathcal{C}$ of continua if, for each $X\in\mathcal{C}$, there exists a continuous surjection from $M$ onto $X$. By the Hahn--Mazurkiewicz theorem (see \cite[Section 8]{nadler1}), the interval $[0,1]$ is a common model for the class of all Peano continua. Similarly, the pseudo-arc is known to serve as a common model for the class of arc-like continua (see \cite{fear}, \cite{lelek}, and \cite{mioduszewski}). On the other hand, using ideas developed in \cite{waraszkiewicz1932}, Waraszkiewicz \cite{waraszkiewicz1934} constructed a family of compactifications of $(0,1]$, each having the circle as remainder, that admits no common model. For further results concerning non-existence of common models, see also \cite{ KrzempekPol2009, MackowiakTymchatyn1984, none, apos, ortega}.

\section{Refinable Maps}

In this section,  we prove the following theorem. The proof is based on \cite[Theorem 2.2]{refinablematsuhashi}.

 \begin{theorem}
Let $(X,d_X)$ and $(Y,d_Y)$ be continua and let $f: X \to Y$ be a refinable map. If $X$ has property $\mathcal{P}_4$, then so does $Y$. 
\label{refinable}
\end{theorem}

\begin{proof}
Let $y\in Y$ and let $\varepsilon>0$. We show that $y$ is a shore point of $Y$. Since $f$ is a refinable map, there exists a sequence $\{f_{\frac{1}{n}}\}_{n=1}^{\infty}$ of $\frac{1}{n}$-refinements of $f$. We may assume that $\lim f_{\frac{1}{n}}^{-1}(y)$ exists. Let $\lim f_{\frac{1}{n}}^{-1}(y)=\{x\}$. Note that there exists $\delta>0$ such that if $A,B\in C(X)$ satisfy $H_{d_X}(A,B)<\delta$, then $H_{d_Y}(f(A),f(B))<\varepsilon$. Since $x$ is a shore point of $X$, there exists a subcontinuum $D$ of $X\setminus \{x\}$ such that $H_{d_X}(D,X)<\delta$. Then, there exists $n_0\in\mathbb{N}$ such that $f_{\frac{1}{n_0}}^{-1}(y)\cap D=\emptyset$ and $d_Y(f(z),f_{\frac{1}{n_0}}(z))<\varepsilon$ for each $z\in X$. It follows that $H_{d_Y}(Y,f_{\frac{1}{n_0}}(D))<2\varepsilon$ and $f_{\frac{1}{n_0}}(D)\subset Y\setminus \{y\}$. Therefore, $y$ is a shore point of $Y$. This completes the proof.
\end{proof}

It was shown in \cite[Theorem 2.2]{refinablematsuhashi} that $\mathcal{P}_1$ is preserved under refinable maps. In Theorem~\ref{refinable}, we proved that $\mathcal{P}_4$ is also preserved under refinable maps. This leads to the following problem.

\begin{problem}
Let $i\in \{2,3,5,6\}$. Is $\mathcal{P}_i$ preserved under refinable maps?
\end{problem}

\section{Whitney properties and Whitney-reversible properties}

In this section, we consider topics concerning Whitney properties and Whitney-reversible properties.

%Our next goal is to give a negative answer to \cite[Question 2.9]{nonweak}. 
The following lemma follows immediately from \cite[Lemma 4.9]{none}.

\begin{lemma}{\rm (\cite[Lemma 4.9]{none}; cf. \cite[Exercise 44.14]{illanes})}
Let $X$ be an indecomposable continuum, and let $a,b,c,d\in X$ lie in distinct composants of $X$. Define $X'$ to be the quotient space obtained from $X$ by identifying $a$ and $b$, and $c$ and $d$, respectively. Then $X'$ is an indecomposable continuum, and for every Whitney map $\mu:C(X')\to[0,\mu(X')]$ and every $t\in(0,\mu(X'))$, the fiber $\mu^{-1}(t)$ is a decomposable continuum.\label{whitneylevel}
\end{lemma}

The following theorem is the main theorem in this section.
 The proof  is based on the idea of the proof of \cite[Lemma 4.9]{none}.

\begin{theorem}
Let $X$ be a hereditarily indecomposable continuum, and let $a, b, c, d \in X$ lie in distinct composants of $X$. Define $X'$ to be the quotient space obtained from $X$ by identifying $a$ and $b$, and $c$ and $d$, respectively. Let $\mu \colon C(X') \to [0, \mu(X')]$ be a Whitney map. Then $X'$ has the following properties:
\begin{enumerate}
\item $\mu^{-1}(t)$ is a decomposable continuum for each $t \in (0, \mu(X'))$;
\item $X'$ is an indecomposable continuum;
\item $\mu^{-1}(t)$ has property $\mathcal{Q}_5$ (i.e., $\mu^{-1}(t)$ has a strong center point) for each $t \in (0, \mu(X'))$;
\item $X'$ does not have property $\mathcal{Q}_3$ (i.e., $X'$ has only non-block points).
\end{enumerate}
\end{theorem}

\begin{proof}
The conditions (1) and (2) follow immediately from Lemma
\ref{whitneylevel}. Since $X'$ is indecomposable, by \cite[p. 241]{noncut}, each point of $X'$ is a non-block point. 

We prove that condition (3) holds. Let $q:X \to X'$ be the quotient
map and let $t \in (0,\mu(X'))$. Define
$\nu:C(X) \to [0,\nu(X)]$ by

$$
\nu(C)=\mu(q(C))
$$

for each $C \in C(X)$. Then $\nu$ is a Whitney map. Let

$$
\mathcal{H}
=
\{\,q(C):C\in\nu^{-1}(t)\,\},
$$

$$
\mathcal{A}
=
\{\,F\in\mu^{-1}(t):q(a)\in F\,\},
\qquad
\mathcal{C}
=
\{\,F\in\mu^{-1}(t):q(c)\in F\,\}.
$$

We have

$$
\mu^{-1}(t)=\mathcal{H}\cup\mathcal{A}\cup\mathcal{C}.
$$

We show that each point of $\mathcal{H}$ is a strong center point of
$\mu^{-1}(t)$. Let $H\in\mathcal{H}$. It is not difficult to see that
$\mathcal{H}$ and $\nu^{-1}(t)$ are homeomorphic, and

$$
\mathcal{A}\cap\mathcal{H}
=
\{q(A_1),q(B_1)\},
$$

and

$$
\mathcal{C}\cap\mathcal{H}
=
\{q(C_1),q(D_1)\},
$$

where $A_1,B_1,C_1,D_1\in\nu^{-1}(t)$ and
$a\in A_1$, $b\in B_1$, $c\in C_1$, and $d\in D_1$.
Note that $A_1$, $B_1$, $C_1$, and $D_1$ are uniquely determined
since $X$ is hereditarily indecomposable.

We prove that there does not exist a proper subcontinuum of
$\mathcal{H}$ that intersects both
$\{q(A_1),q(B_1)\}$ and $\{q(C_1),q(D_1)\}$.
If such a subcontinuum exists, then there exists a proper
subcontinuum $\mathcal{H}'$ of $\nu^{-1}(t)$ that intersects both
$\{A_1,B_1\}$ and $\{C_1,D_1\}$.
We may assume that $\mathcal{H}'$ contains both $A_1$ and $C_1$.
Consequently,

$$
a,c\in\bigcup\mathcal{H}'.
$$

Since $a$ and $c$ belong to distinct composants of $X$, it follows that

$$
\bigcup\mathcal{H}'=X.
$$

However, this contradicts \cite[Theorem 14.14.1]{hyperspaceofsets}.

\smallskip

Note that $\mathcal{A} \setminus \mathcal{H}$ and $\mathcal{C} \setminus \mathcal{H}$ are open in $\mu^{-1}(t)$, and any subcontinuum of $\mu^{-1}(t)$ that intersects both of them must intersect both $\{q(A_1),q(B_1)\}$ and $\{q(C_1),q(D_1)\}$. Hence, at least one component of the intersection of such a continuum with $\mathcal{H}$ must intersect both $\{q(A_1),q(B_1)\}$ and $\{q(C_1),q(D_1)\}$. By the above observation, this component must coincide with $\mathcal{H}$, and hence must contain $H$. Therefore, $H$ is a strong center point of $\mu^{-1}(t)$. It follows that every point of $\mathcal{H}$ is a strong center point of $\mu^{-1}(t)$. This completes the proof. \end{proof}

Hence, we obtain the following result, which gives a negative answer to \cite[Question 2.9]{nonweak}.

\begin{corollary}
The properties $\mathcal{Q}_3$ (having a block point), $\mathcal{Q}_4$ (having a non-shore point), and $\mathcal{Q}_5$ (having a strong center point) are not Whitney-reversible properties.
\label{open1}
\end{corollary}

We can also see the following.

\begin{corollary}{\rm (See also \cite[Theorem 2.6]{nonweak})} 
The properties $\mathcal{P}_3$ (having only non-block points), $\mathcal{P}_4$ (having only shore points), and $\mathcal{P}_5$ (having only non-strong center points) are not Whitney properties.
\label{open2}
\end{corollary}

%\begin{remark}By \cite[Theorem 3.2]{refinablematsuhashi}, being colocally connected is a Whitney property. Hence, having a point of non-colocal connectedness is a Whitney reversible property. Also, in \cite[Theorem 2.2]{matsuhashicut}, it is shown that having a cut point is not a Whitney reversible property. See also \cite[Corollary 2.4]{nonweak} for another result of this type.\end{remark}

The following table summarizes the known results and the results obtained in this paper concerning $\mathcal{P}_i$ and $\mathcal{Q}_i$ and their relations to Whitney properties, Whitney-reversible properties, strong Whitney-reversible properties, and sequential strong Whitney-reversible properties. For each entry, the corresponding reference is given in parentheses; although the reference does not explicitly state the answer as ``yes'' or ``no", the answer can be readily determined from the cited result. The results proved in this paper are shown in red.

\medskip

{\scriptsize \begin{tabular}{c|c|c|c|c}
\hline
  & WP & WRP & SWRP & SSWRP \\
\hline
$\mathcal{P}_1$ &  yes (\cite[Thm. 3.2]{refinablematsuhashi}) &  no (\cite[Ex. 3.3]{refinablematsuhashi}) &  no (\cite[Ex. 3.3]{refinablematsuhashi}) &  no (\cite[Ex. 3.3]{refinablematsuhashi}) \\
$\mathcal{Q}_1$  &   no (\cite[Ex. 3.3]{refinablematsuhashi}) &yes (\cite[Thm. 3.2]{refinablematsuhashi}) & yes (\cite[Thm. 3.2]{refinablematsuhashi}) & yes (\cite[Thm. 3.2]{refinablematsuhashi})\\
\hline
$\mathcal{P}_2$  & yes (\cite[Thm. 2.3]{nonweak}) &  no  (\cite[Ex. 3.3]{refinablematsuhashi}) &  no  (\cite[Ex. 3.3]{refinablematsuhashi}) &  no  (\cite[Ex. 3.3]{refinablematsuhashi})\\
$\mathcal{Q}_2$ &  no  (\cite[Ex. 3.3]{refinablematsuhashi}) &yes(\cite[Cor. 2.4]{nonweak}) & yes(\cite[Cor. 2.4]{nonweak}) & yes(\cite[Cor. 2.4]{nonweak})\\
\hline
$\mathcal{P}_3$  & \textcolor{red}{no (Cor. \ref{open2})}  & no (\cite[Ex. 3.3]{refinablematsuhashi}) & no (\cite[Ex. 3.3]{refinablematsuhashi}) & no (\cite[Ex. 3.3]{refinablematsuhashi})\\
$\mathcal{Q}_3$ & no (\cite[Ex. 3.3]{refinablematsuhashi}) & \textcolor{red}{no (Cor. \ref{open1})}  &  \textcolor{red}{no (Cor. \ref{open1})} & \textcolor{red}{no (Cor. \ref{open1})} \\
\hline
$\mathcal{P}_4$  & \textcolor{red}{no (Cor. \ref{open2})} & no (\cite[Ex. 3.3]{refinablematsuhashi}) &no (\cite[Ex. 3.3]{refinablematsuhashi}) &no (\cite[Ex. 3.3]{refinablematsuhashi})\\
$\mathcal{Q}_4$  &no (\cite[Ex. 3.3]{refinablematsuhashi}) &\textcolor{red}{no (Cor. \ref{open1})}  &  \textcolor{red}{no (Cor. \ref{open1})} & \textcolor{red}{no (Cor. \ref{open1})}\\
\hline
$\mathcal{P}_5$ & \textcolor{red}{no (Cor. \ref{open2})} & no (\cite[Ex. 3.3]{refinablematsuhashi}) & no (\cite[Ex. 3.3]{refinablematsuhashi}) & no (\cite[Ex. 3.3]{refinablematsuhashi})\\
$\mathcal{Q}_5$  & no (\cite[Ex. 3.3]{refinablematsuhashi}) & \textcolor{red}{no (Cor. \ref{open1})}  &  \textcolor{red}{no (Cor. \ref{open1})} & \textcolor{red}{no (Cor. \ref{open1})}\\
\hline
$\mathcal{P}_6$  & no  (\cite[Thm. 2.2]{matsuhashicut}) & no (\cite[Ex. 3.3]{refinablematsuhashi}) & no (\cite[Ex. 3.3]{refinablematsuhashi}) & no (\cite[Ex. 3.3]{refinablematsuhashi})\\
$\mathcal{Q}_6$  & no (\cite[Ex. 3.3]{refinablematsuhashi}) & no  (\cite[Thm. 2.2]{matsuhashicut}) & no  (\cite[Thm. 2.2]{matsuhashicut}) & no  (\cite[Thm. 2.2]{matsuhashicut})\\
\hline
\end{tabular}
}

\section{Non-existence of common models for certain classes of continua}

Here we recall some basic definitions and results from \cite{ortega}, which are based on ideas from \cite[Section 20]{MackowiakTymchatyn1984}. Since the formulations there involve some additional technicalities that are not essential, we present simplified versions of the definitions.

From now on, we denote the closed interval $[0,1]$ by $I$.

%Let $f:X\to Y$ be a map, then we denote the set $\{(x,f(x))  \in X \times Y : x \in X\}$ by $G(f)$. 

%\begin{definition}
%    A continuum \( X \) is said to be {\it arc-like} (respectively, {\it tree-like}) if for every \( \varepsilon > 0 \) there exists a continuous surjection \( f : X \to A \) onto an arc \( A \) (respectively, onto a tree \( T \)) such that for every point \( y \) in \( A \) (respectively, in \( T \)), the diameter of the fiber \( f^{-1}(y) \) is less than \( \varepsilon \).

%\end{definition}

\begin{definition}\label{def:tri-fold}
Let $X \subset I^\infty \times I$ be a continuum, and let 
$A^0, A^1 \subset X$ be nonempty closed subsets such that 
\[
X \cap (I^\infty \times \{0\}) = A^0 
\quad \text{and} \quad
X \cap (I^\infty \times \{1\}) = A^1 .
\]

Define the subsets $T_1^{X,A^0,A^1}, \ T_2^{X,A^0,A^1}$, and 
$T_3^{X,A^0,A^1}$ of $I^\infty \times \mathbb{R} \times \mathbb{R}$ by
\begin{align*}
T_1^{X,A^0,A^1} &= \{(x,t,0) : (x,t) \in X \}, \\
T_2^{X,A^0,A^1} &= \{(x,-t,0) : (x,t) \in X \}, \\
T_3^{X,A^0,A^1} &= \{(x,0,t) : (x,t) \in X \}.
\end{align*}

We define the \emph{tri-fold symmetric extension of $(X,A^0,A^1)$} by
\[
Y_{X,A^0,A^1} := 
T_1^{X,A^0,A^1} \cup T_2^{X,A^0,A^1} \cup T_3^{X,A^0,A^1}.
\]
%Note that
%\[
%T_1^X \cap T_2^X \cap T_3^X = \{(x,0,0) : (x,1) \in A'\} = \{(x,0,0) : (x,-1) \in A''\}.
%\]

\end{definition}

\begin{definition}\label{def:meandering}
Let $X$, $A^0$, $A^1$, and $Y_{X,A^0,A^1}$ be as in the previous definition. 
Define 
\[
Q := I^\infty \times \mathbb{R} \times \mathbb{R},
\]
with natural projections
\[
p: Q \times [0,3] \to Q, 
\quad
q: Q \times [0,3] \to [0,3].
\]

We call a sequence $\mathbf{m} = (m(n))_{n=0}^\infty$ with values in $\{1,2,3\}$
a \emph{meandering pattern} if $m(2n)=m(2n+1)$ for infinitely many $n \ge 0$
and, for each $i \in \{1,2,3\}$, the set $\{\,n \ge 0 : m(n)=i\,\}$ is infinite.

We construct a continuum $M_{X,A^0,A^1}^{\mathbf{m}}$ as follows. There exist:
\begin{itemize}
  \item topological copies $\{X_n\}_{n=0}^\infty$ of $X$, and
  \item a strictly decreasing sequence $(d_n)_{n=1}^\infty \subset (0,1]$ with $d_1=1$ and $\displaystyle \lim_{n\to\infty} d_n = 0$,
\end{itemize}
such that, for each $n\ge0$, we denote by $A_n^0$ and $A_n^1$ the subsets of $X_n$
corresponding to $A^0$ and $A^1$ under the identification of $X_n$ with $X$, and

\begin{enumerate}
  \item $X_0 \subset q^{-1}([1,3])$, with $X_0\cap q^{-1}(3)=A_0^1$ and $X_0\cap q^{-1}(1)=A_0^0$;
  \item for each $n\ge1$, $X_n \subset q^{-1}([d_{n+1},d_n])$;
  \item for each even $n\ge0$, $X_n\cap X_{n+1} = A_n^0 = A_{n+1}^0$;
  \item for each odd $n\ge0$, $X_n\cap X_{n+1} = A_n^1 = A_{n+1}^1$;
  \item for each $n\ge0$, $p|_{X_n}: X_n \to T^{X,A^0,A^1}_{m(n)}$ is a homeomorphism;
  \item the union
    \[
      M_{X,A^0,A^1}^{\mathbf{m}}
      := (Y_{X,A^0,A^1}\times\{0\})\;\cup\;\bigcup_{n=0}^\infty X_n
    \]
    is a continuum.
\end{enumerate}

We call $M_{X,A^0,A^1}^{\mathbf{m}}$ the \emph{meandering continuum with pattern}~$\mathbf{m}$ associated with $({X,A^0,A^1})$, and denote the collection of all such continua by~$\mathbb{M}_{X,A^0,A^1}$.

\end{definition}

  % First, we note from the construction of \( M_X^{\mathbf{m}} \) using (1), (2), and (4) that
%\[
%q^{-1}(\{0\}) = Y_X \cup \{0\} \quad \text{and} \quad q^{-1}((0,3]) = \bigcup_{i=0}^{\infty} Z_i,
%\]
%and thus
%\[
%q^{-1}(0) \cup q^{-1}((0,3]) = M_X^{\mathbf{m}}.
%\]
%First, note that
%$(\pi \circ F)|_{q^{-1}(0)} : Y_X \times \{0\} \to T_1^X$, and  for any sequence \( \{y_i\} \subset T_1^X \) such that \( y_i \to y \in T_1^X \), we can easily see that 
%\[
%((\pi \circ F)|_{q^{-1}(0)})^{-1}(y_i)  \to ((\pi \circ F)|_{q^{-1}(0)})^{-1}(y).
%\]
%This implies that \( (\pi \circ F)|_{q^{-1}(\{0\})} \) is an open map (see \cite[Theorem ?]{nadler1}).

%On the other hand, since $F|_{q^{-1}((0,3])} : q^{-1}((0,3]) \to F(q^{-1}((0,3]))$
%is a homeomorphism and \( \pi  | _{F(q^{-1}((0,3]))} : F(q^{-1}((0,3])) \to T^X_1\) is an open map, it follows that \( (\pi \circ F)|_{q^{-1}((0,3])} : q^{-1}((0,3]) \to T^X_1 \) is also an open map.  

%\begin{lemma}\label{surjection}
%Let \(J=[0,1]\) be the mirror‑glued continuum with 
%\[
%J'=[0,1],\quad J''=[0,1],\quad A'=\{1\},\quad A''=\{0\}.
%\]
%Fix a meandering pattern \(\mathbf{m}=(m(i))_{i=0}^\infty\subset \{1,2,3\}\), and let 
%\(M_J^{\mathbf{m}}\in\mathbb{M}_J\) be the corresponding continuum. Then for any mirror‑glued continuum \(X=(X',A')\textrm{–}(X'',A'')\), there exist 
%\[
%M_X^{\mathbf{m}}\in\mathbb{M}_X
%\quad\text{and}\quad
%f\colon M_X^{\mathbf{m}}\to M_J^{\mathbf{m}}
%\]
%a continuous surjection.
%\end{lemma}

The following result can be obtained by a proof completely similar to that of \cite[Corollary 2.8]{ortega}.

\begin{theorem}{\rm (see also \cite[Corollary 2.8]{ortega})}
\label{cor:no-common-model} Let $X$, $A^0$ and $A^1$ be as in  Definition \ref{def:tri-fold}. 
 Then, the family $\mathbb{M}_{X,A^0,A^1}$ of meandering continua has no common model.\label{ortega}
\end{theorem}
    
The notation introduced in this section will be used throughout the subsequent sections without further explanation.

\begin{remark}\label{meandering}

In the original construction of \cite{MackowiakTymchatyn1984}, each of the three branch indices $i \in \{1,2,3\}$ naturally appears infinitely often in the sequence. Although the formal definition of a meandering pattern in \cite{ortega} omitted the explicit requirement that $\{\,n \ge 0 : m(n)=i\,\}$ be infinite for each $i \in \{1,2,3\}$, explicitly imposing this condition in Definition \ref{def:meandering} aligns with the original construction in \cite{MackowiakTymchatyn1984} and does not affect the validity of Theorem \ref{cor:no-common-model}.
\end{remark}

\begin{remark}
Condition~(1) in Definition \ref{def:meandering} is inessential and is imposed only to simplify the argument in \cite{ortega}, which is based on \cite[Section 20]{MackowiakTymchatyn1984}. Therefore, the figures of meandering continua appearing later in this paper do not reflect this condition.

\end{remark}

%In this section, we 
%prove that the class of continua that witnesses the failure of the Whitney property for the property of consisting entirely of shore points has no common model 

\begin{theorem}
\label{shore}
Let $\mathcal{C}$ be the class of all continua $Z$ satisfying the following properties:
\begin{enumerate}
\item for each Whitney map $\mu \colon C(Z) \to [0, \mu(Z)]$ and for each sufficiently small $t \in (0, \mu(Z))$, $\mu^{-1}(t)$ has property $\mathcal{Q}_5$ (i.e., $\mu^{-1}(t)$ contains a strong center point), and
\item $Z$ does not have property $\mathcal{Q}_3$ (i.e., $Z$ has only non-block points).
\end{enumerate}
Then $\mathcal{C}$ does not admit a common model.
\end{theorem}

\begin{proof}
The idea of the proof   is taken from \cite[Example~2.7]{nonweak}. %Before proceeding, we note that some of the notation appearing in this example is explained in Section~2.
 Let $X$ be an arc. We may assume that $X \subset I^\infty \times I$ and that the endpoints of $X$ lie in $I^\infty \times \{0\}$ and $I^\infty \times \{1\}$.  
Let $A^0=X \cap (I^\infty \times \{0\})$ and $A^1=X \cap (I^\infty \times \{1\})$. 
We may further arrange the construction so that $X$ is a line segment. So far we have written $M^{\mathbf{m}}_{X,A^0,A^1}$ to indicate the dependence on $A^0$ and $A^1$. 
Since $X$ is a line segment and $A^0$, $A^1$ are its endpoints, we omit $A^0$ and $A^1$ from the notation from now on. The same convention will be applied to all other notations involving $A^0$ and $A^1$.

Let $\mathbf{m}$ be a meandering pattern. 
We slightly thicken $q^{-1}((0,3]) \cap M^\mathbf{m}_{X}$ and replace it by the product
$B_{X}^{\mathbf m} := (q^{-1}((0,3]) \cap M^\mathbf{m}_{X}) \times I$.
We may assume that $B_{X}^{\mathbf m} \subset Q \times [0,3]$.
 Here, $B_{X}^{\mathbf{m}}$ is a band along  $q^{-1}((0,3])\cap M^\mathbf{m}_{X}$ whose width decreases as the height goes downward, and the closure of $B^{\mathbf{m}}_{X}$ is equal to $ (Y_{X} \times \{0\}) \cup B_{X}^{\mathbf{m}} $. Moreover, let $B'^{\mathbf{m}}_{X}$ denote the manifold boundary of $B_{X}^{\mathbf{m}}$. 
Then $B'^{\mathbf{m}}_{X}$ is homeomorphic to the real line, and note that the closure of $B'^{\mathbf{m}}_{X}$ is equal to $(Y_{X} \times \{0\}) \cup B'^{\mathbf{m}}_{X}$. We can write $B'^{\mathbf{m}}_{X}=P_X^{\mathbf{m}} \cup Q_X^{\mathbf{m}}$, where both $P_X^{\mathbf{m}}$ and $Q_X^{\mathbf{m}}$ are homeomorphic to $[0, \infty)$ and $|P_X^{\mathbf{m}} \cap Q_X^{\mathbf{m}}|=1$.  Let $ Z_{X}^{\mathbf{m}}=(Y_{X} \times \{0\}) \cup B'^{\mathbf{m}}_{X}$ 
 (see Figure~\ref{M}). It is easy to see that $Z_X^{\mathbf{m}}$ has only non-block points. 
 
 \begin{figure}[ht]
\includegraphics[height=6cm]{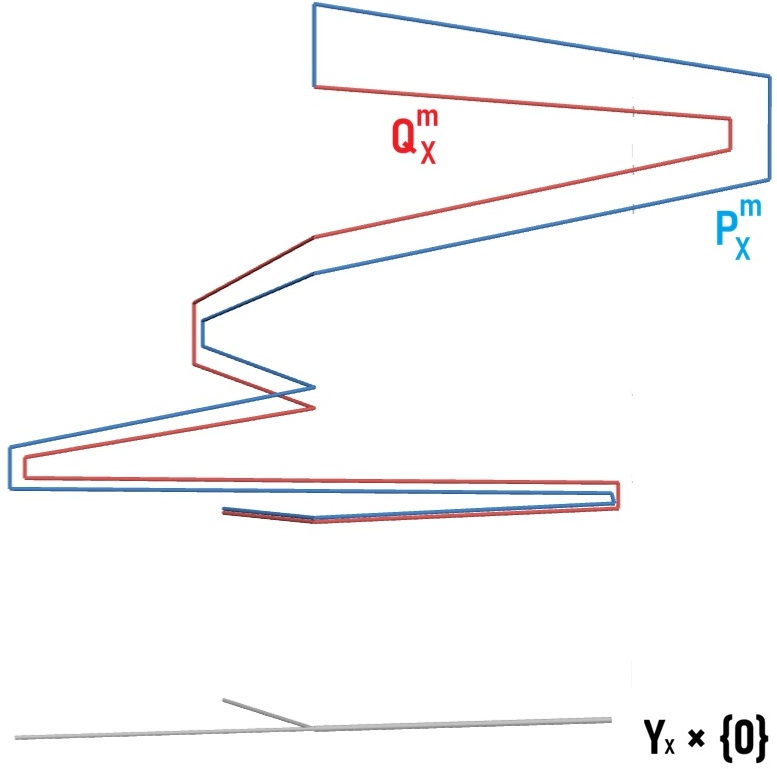}
\caption{$Z^\mathbf{m}_X$}
\label{M} 
\end{figure}
 %Also, we may assume that there exists a continuous surjection from $Z_X^\mathbf{m}$ onto $M_X^\mathbf{m}$  that 
 %is the identity on $Y_X \times \{0\}$.
%(For example, by choosing $B'^{\mathbf{m}}_X$ to be a broken line as illustrated in the figure, one can construct a continuous map from $Z_X^\mathbf{m} $ to $M_X^\mathbf{m}$ that is the identity on $Y_X$). 
 
Let $\mu : C(Z_X^{\mathbf{m}}) \to [0,\mu(Z_X^{\mathbf{m}})]$ be a Whitney map. Let %$A = \{(x,t,0) \in T^X_1 \mid \frac{1}{2} \le t \le \frac{3}{4}\}$
 $t_0 = \min \{\mu(T^X_j \times \{0\}): j\in\{1,2,3\}\}$.  We may assume that $t_0=\mu(T^X_1 \times \{0\})$. Fix a sufficiently small $t \in (0, t_0)$ and a simple triod $T \subset Y_X$ satisfying $\mu(T \times \{0\}) = t$. Write $T = J_1 \cup J_2 \cup J_3$, where $J_j \subset T^X_j$ for $j \in \{1,2,3\}$.  Put $\varepsilon_1 = \frac{1}{2} \min \{{\rm diam} J_1, {\rm diam} J_2, {\rm diam} J_3, {\rm diam} (T^X_1 \setminus J_1), {\rm diam} (T^X_2 \setminus J_2),  {\rm diam} (T^X_3 \setminus J_3)\}$. 
Then, there exists $s_0 > 0$ such that if $C \in C(Z_X^{\mathbf{m}})$ satisfies  $C \subset q^{-1}((0,s_0])$, then $H_d(C,T
\times \{0\}) > \varepsilon_1$, where $d$ is an admissible metric on $Z_X^{\mathbf{m}}$ and $H_d$ is the Hausdorff metric on $C(Z_X^{\mathbf{m}})$ that is induced by $d$. Also, there exists $\varepsilon_2 > 0$ such that if $C \in C(Z_X^{\mathbf{m}})$ satisfies  $C \cap q^{-1}([s_0,3]) \neq \emptyset $, then $H_d(C,T
\times \{0\}) \ge \varepsilon_2$. %Furthermore, there exists $\varepsilon_3 > 0$ such that if a subcontinuum $ \mathcal{A} \subset  \mu^{-1}(t_0)$ satisfies  $H_{H_d}(\mathcal{A}, \mu^{-1}(t_0)) < \varepsilon_3$, then there exists $C \in \mathcal{A}$ such that $C \cap q^{-1}((0,3]) \neq \emptyset$.

Put $\varepsilon= \min \{\varepsilon_1, \varepsilon_2\}$. Let $\mathcal{U}=\{C \in \mu^{-1}(t) : C \subset q^{-1}((0,3])\}$ and $\mathcal{V}=\{C \in \mu^{-1}(t) : H_d(T \times \{0\}, C) < \varepsilon\}$. It is easy to see that $\mathcal{U}$ and $\mathcal{V}$ are open subsets of $\mu^{-1}(t)$. We show that if $\mathcal{A}$ is a subcontinuum of  $\mu^{-1}(t)$  such that $\mathcal{U} \cap \mathcal{A} \neq \emptyset \neq \mathcal{V} \cap \mathcal{A}$, then $L_1^X \times \{0\} \in \mathcal{A}$, where $L_1^X$ is the subarc of $T_1^X$ such that $\mu(L_1^X \times \{0\})=t$ and $L_1^X$ contains an end point of $Y_X$.

 Then, it is clear that  $\bigcup \mathcal{A} \cap q^{-1}((0,3]) \neq \emptyset$. Also, by the choice of 
 $\varepsilon$, $\bigcup \mathcal{A} \cap q^{-1}(0) \neq \emptyset$. %However, note that there does not exists an element of $\mathcal{A}$ that intersects both of $q^{-1}(0)$ and $q^{-1}((0,3])$.
Since $\bigcup \mathcal{A}$ is connected, we see that  there exists a sequence $\{A_i\}_{i=1}^\infty$ of $\mathcal{A}$ and a decreasing sequence
$\{c_i\}_{i=1}^\infty \subset (0,3]$ which converges to 0 such that $A_i \cap q^{-1}(c_i) \neq \emptyset$ for each $i \ge 1$. Then, $A_i \subset q^{-1}((0,3])$ for each $i > 0$. 
Although an individual $A_i$ may intersect both $P_X^{\mathbf m}$ and $Q_X^{\mathbf m}$,
such continua cannot occur infinitely often in a sequence that remains in $\mu^{-1}(t)$ while descending to level $0$.
Indeed, any $A_i$ extending from the intersection $P_X^{\mathbf m} \cap Q_X^{\mathbf m}$ to near level $0$ would spread along the meandering structure to approximate the shape of $Y_X$, forcing $\mu(A_i) \ge t_0 > t$.
Hence, by passing to a subsequence if necessary, we may assume that $A_i \subset P_X^{\mathbf m}$ for all $i \ge 1$.

Since $P^{\mathbf{m}}_X$ is homeomorphic to $[0,\infty)$, each component of $\{C \in \mathcal{A} : C \subset P^{\mathbf{m}}_X\}$ is homeomorphic to a one-point set or an interval. If there exist three different components of $\{C \in \mathcal{A} : C \subset P^{\mathbf{m}}_X\}$, then one of them is separated from $\{C \in \mathcal{A} : C \cap Q^{\mathbf{m}}_X \neq \emptyset\}$ and $\{C \in \mathcal{A} : C \subset q^{-1}(0)\}$. Therefore, this component is also a component of $\mathcal{A}=\{C \in \mathcal{A} : C \subset P^{\mathbf{m}}_X\} \cup \{C \in \mathcal{A} : C \cap Q^{\mathbf{m}}_X \neq \emptyset\} \cup \{C \in \mathcal{A} : C \subset q^{-1}(0)\}$. This contradicts the fact that $\mathcal{A}$ is connected. Hence, the number of components of $\{C \in \mathcal{A} : C \subset P^{\mathbf{m}}_X\}$ is at most two. Thus, we may assume that for each $i \ge 1$, there exists a subarc $\mathcal{I}_i$ of $\{C \in \mathcal{A} : C \subset P^{\mathbf{m}}_X\}$ that joins $A_i$ and $A_{i+1}$. In view of the meandering shape of $P_X^{\mathbf m}$ and \cite[Exercise 4.33 (b)]{nadler1}, for infinitely many $i \ge 1$, the subarc $\mathcal I_i$ joining $A_i$ and $A_{i+1}$ contains a continuum $A'_i \in \mu^{-1}(t)$ located near the turn of $P_X^{\mathbf m}$ such that $A'_i \to L_1^X \times \{0\}$. 
Hence, $L^X_1 \times \{0\} \in \mathcal A$.  %which is a contradiction.
%Hence, there exists $s_1 \in (0,3]$ such that $\bigcup \mathcal{A} \subset q^{-1}(0) \cup q^{-1}([s_1,3])$. Since $\mathcal{A}$ is a continuum,
%$\bigcup \mathcal{A} \subset q^{-1}(0)$ or $q^{-1}([s_1,3])$.
%However, both contradict to $H_{H_d}(\mathcal{A}, \mu^{-1}(t_0)) < \varepsilon$. 
 Thus, $L^X_1 \times \{0\}$ is  a strong  center  point of $\mu^{-1}(t)$. Therefore, $Z_X^\mathbf{m} \in \mathcal{C}$ and we see that $\{ Z_X^\mathbf{m} : \mathbf{m} \text{ is a meandering pattern} \} \subset \mathcal{C}$.

The space $M_X^\mathbf{m}$ is obtained from $Z_X^\mathbf{m}$ by collapsing the boundary of the parts
thickened along the meandering pattern onto $M_X^\mathbf{m}$; this collapse yields a
continuous surjection which is the identity on $Y_X \times \{0\}$ for each
meandering pattern $\mathbf{m}$. Hence, together with Theorem \ref{cor:no-common-model}, we see that the family
$\{ Z_X^\mathbf{m} : \mathbf{m} \text{ is a meandering pattern}\}$ admits no common model. Therefore, neither does $\mathcal{C}$.
  \end{proof}

 \begin{corollary}
For each property $\mathcal{Q}_3$ (having a block point), $\mathcal{Q}_4$ (having a non-shore point), and $\mathcal{Q}_5$ (having a strong center point), the class of continua witnessing the failure of the sequential strong Whitney-reversible property admits no common model.
\label{57}
\end{corollary}

\begin{corollary}
For each property $\mathcal{P}_3$ (having only non-block points), $\mathcal{P}_4$ (having only shore points), and $\mathcal{P}_5$ (having only non-strong center points), the class of continua witnessing the failure of the Whitney property admits no common model.
\label{58}
\end{corollary}

 Recall that $\mathcal{Q}_1$ and $\mathcal{Q}_2$ are Whitney-reversible (indeed, sequentially strong Whitney-reversible) properties (see \cite[Theorem 3.2]{refinablematsuhashi} and \cite[Corollary 2.4]{nonweak}), so the classes of continua witnessing the failure of the Whitney-reversible property (resp., strong Whitney-reversible property, sequential strong Whitney-reversible property) for $\mathcal{Q}_1$ and $\mathcal{Q}_2$ are empty. 
 On the other hand, for $\mathcal{Q}_6$ (having a cut point), it was already established in \cite[Theorem 4.1]{none} that the corresponding witnessing classes for this property  admit no common model. Thus, among $\mathcal{Q}_1, \ldots, \mathcal{Q}_6$, the question of whether the witnessing class for the Whitney-reversible property (resp., strong Whitney-reversible property) admits a common model remains open solely for $\mathcal{Q}_3, \mathcal{Q}_4$, and $\mathcal{Q}_5$. Hence, we naturally pose the following question.
\begin{problem}
\label{open_prob}
For each $i \in \{3, 4, 5\}$, let $\mathcal{C}_i$ (resp., $\mathcal{C}_i'$) be the class of continua witnessing the failure of the Whitney-reversible property (resp., strong Whitney-reversible property) for $\mathcal{Q}_i$. Does $\mathcal{C}_i$ (or $\mathcal{C}_i'$) admit a common model?
\label{problemcommon}
\end{problem}

  A continuum $X$ is called \emph{aposyndetic} if, for every two distinct points $x,y\in X$, there exists a subcontinuum $K$ of $X$ such that $x\in\operatorname{Int}_X K$ and $y\notin K$. 
It is known that being colocally connected and being aposyndetic are not Whitney-reversible properties; see \cite[Exercise 29.9]{illanes} and \cite[Example 3.3]{refinablematsuhashi}. Note that every colocally connected continuum is aposyndetic, and every point of a colocally connected continuum is a non-cut point. In \cite[Example 3.2]{davidm}, an example is given of a continuum for which every positive Whitney level is colocally connected, while the continuum itself is not aposyndetic and has a cut point. Although it is not explicitly stated in \cite[Example 3.2]{davidm} that the continuum has a cut point, this is readily seen from the construction of the example. We now strengthen the result as follows.%As can be seen from the proof of the following theorem,  Theorem \ref{cor:no-common-model} and \cite[Corollary 3.4]{david} readily imply that the class of such continua does not have a common model.

\begin{theorem}
Let $\mathcal{C}$ be the class of continua $Z$ satisfying the following properties:
\begin{enumerate}
\item for each Whitney map $\mu \colon C(Z) \to [0, \mu(Z)]$ and for each $t \in (0, \mu(Z))$, the Whitney level $\mu^{-1}(t)$ has property $\mathcal{P}_1$ (i.e., $\mu^{-1}(t)$ is colocally connected), and
\item $Z$ does not have property $\mathcal{P}_6$ (i.e., $Z$ has a cut point) and is not aposyndetic.
\end{enumerate}
Then $\mathcal{C}$ does not admit a common model.
\label{P}
\end{theorem}

\begin{proof}
Let $X$ be a simple closed curve. We may assume that $X \subset I^\infty \times I$ and that $|X \cap (I^\infty \times \{0\})|=|X \cap (I^\infty \times \{1\})|=1$. Let $A^0=X\cap(I^\infty\times\{0\})$ and $A^1=X\cap(I^\infty\times\{1\})$. Then it is easy to see that, for each meandering pattern $\mathbf{m}$, $M_{X,A^0,A^1}^{\mathbf m}$ has a cut point and is not aposyndetic. 

We show that, for each $t\in(0,\mu(M_{X,A^0,A^1}^{\mathbf m}))$, $\mu^{-1}(t)$ is colocally connected. Let $A\in\mu^{-1}(t)$. Let $J$ denote the closure of the set of all junction points of the simple closed curves in $M_{X,A^0,A^1}^{\mathbf m}$. It is easy to see that $A \setminus J\neq\emptyset$ and that each point of $A \setminus J$ is a point of colocal connectedness of $M_{X,A^0,A^1}^{\mathbf m}$. Hence, by \cite[Corollary 3.4]{david}, $A$ is a point of colocal connectedness of $\mu^{-1}(t)$. Therefore, $\mu^{-1}(t)$ is colocally connected.  Together with Theorem \ref{cor:no-common-model}, we see that the family $\{M_{X,A^0,A^1}^{\mathbf m}:\mathbf{m}\text{ is a meandering pattern}\}$ admits no common model. Therefore, neither does $\mathcal{C}$.
\end{proof}

\begin{remark}
Regarding the properties $\mathcal{P}_1, \ldots, \mathcal{P}_6$, the previous theorem implies that for each $i \in \{1, \ldots, 6\}$, the class of continua witnessing the failure of the Whitney-reversible property for $\mathcal{P}_i$ already admits no common model. Thus, the corresponding question for $\mathcal{P}_i$ is completely resolved.
\end{remark}

Finally, we summarize in the following  table the main results regarding the common model problem for the properties $\mathcal{P}_i$ and $\mathcal{Q}_i$ ($i = 1, \ldots, 6$). Here, ``not WP (CM)" (resp., ``not WRP (CM)", ``not SWRP (CM)", ``not SSWRP (CM)") denotes whether the class of continua witnessing the failure of the Whitney property (resp., Whitney-reversible property, strong Whitney-reversible property, sequential strong Whitney-reversible property) admits a common model. The entry `` $\emptyset$ " indicates that the corresponding property is indeed a Whitney or Whitney-reversible property, so the class witnessing its failure is empty. The entry ``no" means that the corresponding witnessing class admits no common model, whereas ``open" signifies that the common model problem remains unresolved (see Problem~\ref{problemcommon}). The results established in this paper are indicated in red.

\medskip

{\scriptsize 
\begin{tabular}{c|c|c|c|c}
\hline
  & not WP (CM) & not WRP (CM) & not SWRP (CM) & not SSWRP (CM) \\
\hline
$\mathcal{P}_1$ & $\emptyset$  (\cite[Thm. 3.2]{refinablematsuhashi})& \Rno{Thm. \ref{P}} & \Rno{Thm. \ref{P}} & \Rno{Thm. \ref{P}} \\
$\mathcal{Q}_1$ & \Rno{Thm. \ref{P}} & $\emptyset$  (\cite[Thm. 3.2]{refinablematsuhashi}) & $\emptyset$  (\cite[Thm. 3.2]{refinablematsuhashi}) & $\emptyset$  (\cite[Thm. 3.2]{refinablematsuhashi}) \\
\hline
$\mathcal{P}_2$ &  $\emptyset$   (\cite[Thm. 2.3]{nonweak})  & \Rno{Thm. \ref{P}} & \Rno{Thm. \ref{P}} & \Rno{Thm. \ref{P}} \\
$\mathcal{Q}_2$ & \Rno{Thm. \ref{P}} &  $\emptyset$   (\cite[Cor. 2.4]{nonweak})  &  $\emptyset$  (\cite[Cor. 2.4]{nonweak})  &  $\emptyset$  (\cite[Cor. 2.4]{nonweak}) \\
\hline
$\mathcal{P}_3$ & \Rno{Cor. \ref{58}} &  \Rno{Thm. \ref{P}} &  \Rno{Thm. \ref{P}} &  \Rno{Thm. \ref{P}}\\
$\mathcal{Q}_3$ & \Rno{Thm. \ref{P}}& open & open & \Rno{Cor. \ref{57}} \\
\hline
$\mathcal{P}_4$ & \Rno{Cor. \ref{58}} &  \Rno{Thm. \ref{P}} &  \Rno{Thm. \ref{P}} &  \Rno{Thm. \ref{P}} \\
$\mathcal{Q}_4$ &\Rno{Thm. \ref{P}} & open & open & \Rno{Cor. \ref{57}} \\
\hline
$\mathcal{P}_5$ & \Rno{Cor. \ref{58}} &  \Rno{Thm. \ref{P}} &  \Rno{Thm. \ref{P}} &  \Rno{Thm. \ref{P}}\\
$\mathcal{Q}_5$ & \Rno{Thm. \ref{P}}& open & open & \Rno{Cor. \ref{57}} \\
\hline
$\mathcal{P}_6$ & no (\cite[Thm. 4.1]{none}) & \Rno{Thm. \ref{P}}&  \Rno{Thm. \ref{P}} &  \Rno{Thm. \ref{P}}\\
$\mathcal{Q}_6$ & \Rno{Thm. \ref{P}} &  no (\cite[Thm. 4.1]{none}) &  no (\cite[Thm. 4.1]{none}) &no (\cite[Thm. 4.1]{none}) \\
\hline

\label{summary}
\end{tabular}
}

\bibliographystyle{amsplain}
\bibliography{refs}

%\bibitem{anderson} R. D. Anderson, {\it Atomic decompositions of continua}, Duke Math. J., {\bf 23} (1956), 507-514.

\end{document}